\documentclass[twoside]{article}
\usepackage{amsmath,amssymb,amscd,amsthm,verbatim,alltt,amsfonts,array}
\usepackage{mathrsfs}
\usepackage[english]{babel}
\usepackage{latexsym}
\usepackage{amssymb}
\usepackage{euscript}
\usepackage{graphicx}

\def\boxit#1{\vbox{\hrule\hbox{\vrule\kern2pt
 \vbox{\kern4pt#1\kern2pt}\kern2pt\vrule}\hrule}}
\def\qed{\hfill
 $\hskip0.3cm
 \boxit{\hsize 2pt \vsize 10pt}$\bigskip\noindent}

\def\C{{\mathchoice {\setbox0=\hbox{$\displaystyle\rm C$}\hbox{\hbox
to0pt{\kern0.4\wd0\vrule height0.9\ht0\hss}\box0}}
{\setbox0=\hbox{$\textstyle\rm C$}\hbox{\hbox
to0pt{\kern0.4\wd0\vrule height0.9\ht0\hss}\box0}}
{\setbox0=\hbox{$\scriptstyle\rm C$}\hbox{\hbox
to0pt{\kern0.4\wd0\vrule height0.9\ht0\hss}\box0}}
{\setbox0=\hbox{$\scriptscriptstyle\rm C$}\hbox{\hbox
to0pt{\kern0.4\wd0\vrule height0.9\ht0\hss}\box0}}}}
\def\Q{{\mathchoice {\setbox0=\hbox{$\displaystyle\rm
Q$}\hbox{\raise
0.15\ht0\hbox to0pt{\kern0.4\wd0\vrule height0.8\ht0\hss}\box0}}
{\setbox0=\hbox{$\textstyle\rm Q$}\hbox{\raise
0.15\ht0\hbox to0pt{\kern0.4\wd0\vrule height0.8\ht0\hss}\box0}}
{\setbox0=\hbox{$\scriptstyle\rm Q$}\hbox{\raise
0.15\ht0\hbox to0pt{\kern0.4\wd0\vrule height0.7\ht0\hss}\box0}}
{\setbox0=\hbox{$\scriptscriptstyle\rm Q$}\hbox{\raise
0.15\ht0\hbox to0pt{\kern0.4\wd0\vrule height0.7\ht0\hss}\box0}}}}
\def\T{{\mathchoice {\setbox0=\hbox{$\displaystyle\rm
T$}\hbox{\hbox to0pt{\kern0.3\wd0\vrule height0.9\ht0\hss}\box0}}
{\setbox0=\hbox{$\textstyle\rm T$}\hbox{\hbox
to0pt{\kern0.3\wd0\vrule height0.9\ht0\hss}\box0}}
{\setbox0=\hbox{$\scriptstyle\rm T$}\hbox{\hbox
to0pt{\kern0.3\wd0\vrule height0.9\ht0\hss}\box0}}
{\setbox0=\hbox{$\scriptscriptstyle\rm T$}\hbox{\hbox
to0pt{\kern0.3\wd0\vrule height0.9\ht0\hss}\box0}}}}
\def\Z{{\mathchoice {\hbox{$\sf\textstyle Z\kern-0.4em Z$}}
{\hbox{$\sf\textstyle Z\kern-0.4em Z$}}
{\hbox{$\sf\scriptstyle Z\kern-0.3em Z$}}
{\hbox{$\sf\scriptscriptstyle Z\kern-0.2em Z$}}}}

\newcommand{\eqnoone}  
   {}
\newcommand{\eqnotwo}  
   {}

\newcounter{alf}

\def\Z{\Bbb Z}

\def\C{\Bbb C}

\newtheorem{theorem}{Theorem}
\theoremstyle{plain}

\newtheorem{definition}{Definition}

\newtheorem{lemma}{Lemma}

\newtheorem{problem}{Problem}
\newtheorem{proposition}{Proposition}
\newtheorem{remark}{Remark}

\numberwithin{equation}{section}

\begin{document}

  \vskip 1.2 true cm

\begin{center} {\bf Strongly order continuous operators on Riesz spaces } \\ 
          {by}\\
{\sc Akbar Bahramnezhad and Kazem Haghnejad Azar}
\end{center}

\pagestyle{myheadings}
\markboth{Strongly order continuous operators}{A. Bahramnezhad and K. Haghnejad Azar}
\begin{center}
 Department of Mathematic, University of Mohaghegh Ardabili, Ardabil,  Iran 
\end{center}
\begin{center}
E-mail: {bahramnezhad@uma.ac.ir}
\end{center}

\begin{abstract}
	In this paper we introduce two  new classes of operators that we call strongly order continuous and strongly  $\sigma$-order continuous operators.  An operator $T:E\rightarrow F$ between two Riesz spaces is said to be strongly  order continuous (resp. strongly $\sigma$-order continuous), if $x _\alpha \xrightarrow{uo}0$ (resp. $x _n \xrightarrow{uo}0$) in $E$ implies $Tx _\alpha \xrightarrow{o}0$ (resp. $Tx _n \xrightarrow{o}0$) in $F$. We give some conditions under which order continuity will be equivalent to strongly order continuity of  operators on Riesz spaces. We show that the collection of all  $so$-continuous linear functionals on a Riesz space $E$ is a band of $E^\sim$.
\end{abstract}

\begin{quotation}

\end{quotation}

\thispagestyle{empty}

\section{Introduction}\label{Sec1}

The concept of unbounded order convergence or $uo$-convergence was introduced in \cite{8} and is proposed firstly  in \cite{3}. It  has recently been intensively studied in several papers \cite{4,5,6}.  Recall that a net $(x_\alpha)_{\alpha\in \mathcal{A}}$ in a Riesz space  $E$ is  {\it  order convergent} (or, $o$-convergent for short)  to  $x\in E$, denoted by $x _\alpha \xrightarrow{o}x$ whenever there exists another net $(y_\beta)_{\beta\in \mathcal{B}}$  in $E$ such that $y_\beta \downarrow 0$ and that for every $\beta\in \mathcal{B}$, there exists $\alpha_0\in \mathcal{A}$ such that  $|x_\alpha -x| \leq y_\beta$ for all $\alpha\geq \alpha_0$.   A net  $(x_\alpha)$ in a Riesz  space $E$ is  {\it  unbounded  order convergent }(or,  $uo$-convergent for short) to $x\in E$ if $|x_\alpha -x|\wedge u\xrightarrow{o}0$ for all $u\in E^+$. We denote this convergence by  $x _\alpha \xrightarrow{uo}x$ and write that $x_\alpha$ $uo$-convergent to $x$. This is an analogue of pointwise convergence in function spaces. Let $\mathbb{R}^A$ be the Riesz space of all real-valued functions on a non-empty set $A$, equipped with the pointwise order. It is easily seen that a net $(x_\alpha)$ in $\mathbb{R}^A$ $uo$-converges to $x\in \mathbb{R}^A$ if and only if it converges pointwise to $x$. For instance in $c_0$ and $\ell_p$($1\leq p\leq \infty)$, $uo$-convergence of nets is the same as coordinate-wise convergence.  Assume that $(\Omega,\Sigma,\mu)$ is a measure space and let $E=L_p(\mu)$ for some $1\leq p<\infty$. Then $uo$-convergence of sequences in $L_p(\mu )$  is the same as almost everywhere convergence. In \cite{9}, Wickstead characterized the spaces in which weak convergence of nets implies $uo$-convergence and vice versa.  In \cite{4},  Gao characterized the space $E$ such that in its dual space $E^*$, $uo$-convergence implies $w^*$-convergence and vice versa. He also characterized the spaces in whose dual space simultaneous $uo$- and $w^*$-convergence imply weak/norm convergence. We show that the collection of all order bounded strongly order continuous linear functionals on a Riesz space $E$  is a band of $E^\sim$ [Theorem 2.7]. For unexplained terminology and facts on  Banach lattices and positive operators, we refer the reader to \cite{1,2}.  Let us start with the definition.
 
 \begin{definition}
An operator $T:E\rightarrow F$ between two Riesz spaces is said to be:
\begin{enumerate}
\item  Strongly order continuous (or, $so$-continuous for short), if $x _\alpha \xrightarrow{uo}0$  in $E$ implies $Tx _\alpha \xrightarrow{o}0$  in $F$.
\item Strongly $\sigma$-order continuous (or, $s\sigma o$-continuous for short), if $x _n \xrightarrow{uo}0$  in $E$ implies $Tx _n \xrightarrow{o}0$  in $F$.
\end{enumerate}
\end{definition}

The collection of all $so$-continuous operators of $L_b(E,F)$  will be denoted by $L_{so}(E,F)$, that is $$L_{so}(E,F):=\{T\in L_b(E,F)  : \mbox{T  is  $so$-continuous}\}.$$
Similarly, $L_{s\sigma o}(E,F)$ will denote the collection of all order bounded operators from $E$ to $F$ that are $s\sigma o$-continuous. That is, $$L_{s\sigma o}(E,F):=\{T\in L_b(E,F) : \mbox{T  is  $s\sigma o$-continuous}\}.$$
Clearly,  we have  $L_{so}(E,F)\subset  L_{s\sigma o}(E,F)$ and   $L_{so}(E,F)$ and  $L_{s\sigma o}(E,F)$ are both vector subspaces of $L_{b}(E,F)$.

Recall that an operator $T:E\rightarrow F$ between two Riesz spaces is said to be  order continuous (resp. $\sigma$-order continuous) if $x _\alpha \xrightarrow{o}0$ (resp. $x _n \xrightarrow{o}0$)  in $E$ implies $Tx _\alpha \xrightarrow{o}0$ (resp. $Tx _n \xrightarrow{o}0$)  in $F$. The collection of all order continuous operators of $L_b(E,F)$  will be denoted by $L_{n}(E,F)$, that is $$L_{n}(E,F):=\{T\in L_b(E,F)  : \mbox{T  is  order continuous}\}.$$
Similarly, $L_c(E,F)$ will denote the collection of all order bounded operators from $E$ to $F$ that are $\sigma$-order continuous. That is, $$L_c(E,F):=\{T\in L_b(E,F) : \mbox{T  is  $\sigma$-order continuous}\}.$$

 Note that every $so$-continuous (resp. $s\sigma o$-continuous ) operator is order (resp. $\sigma$-order) continuous. The converse is not true in general. For example the identity operator $I:c_0\rightarrow c_0$ is order continuous but is not $so$-continuous. Indeed, the standard basis sequence of $c_0$ is $uo$-converges to 0 but is not order convergent.

\section{Main Results}

\begin{lemma} (\cite[Lemma 3.1]{5}). 
In a Riesz space we have the following:
\begin{enumerate}
\item If $x _\alpha \xrightarrow{uo}x$ and $x _\alpha \xrightarrow{uo}y$, then $x=y$. In other hands, unbounded order limits are uniquely determined.
\item If $x _\alpha \xrightarrow{uo}x$, $y _\alpha \xrightarrow{uo}y$ and $k,r$ are real numbers, then $kx _\alpha +ry_\alpha \xrightarrow{uo}kx+ry$. Furthermore  $x_\alpha\vee y_\alpha  \xrightarrow{uo}x\vee y$ and $x_\alpha\wedge y_\alpha  \xrightarrow{uo}x\wedge y$. In particular $x _\alpha ^+ \xrightarrow{uo}x^+$, $x _\alpha ^- \xrightarrow{uo}x^-$ and $|x _\alpha| \xrightarrow{uo}|x|$.
\item If $x _\alpha \xrightarrow{uo}x$ and $x_\alpha \geq y$ holds for all $\alpha$, then $x\geq y$. 

\end{enumerate}
\end{lemma}

Note that the $uo$-convergence in a Riesz space  $E$ does not necessarily correspond to a topology on $E$. For example, let  $E=c$,  the Banach lattice of real valued convergent sequences. Put $x_n=\Sigma_{k=1}^n e_k$, where $(e_n)$ is the standard basis. Then $(x_n)$ is $uo$-convergent to $x=(1,1,1,...)$, but it is not norm convergent.

\begin{proposition}
Let $E$, $F$ be Riesz spaces such that  $E$ is finite-dimensional. Then $L_{so}(E,F)= L_n(E,F)$ and $L_{s\sigma o}(E,F)= L_c(E,F)$. 
\end{proposition}

\begin{proof}
Follows immediately if we observe that in a finite-dimensional Riesz space order convergence  is equivalent to $uo$-convergence.
\end{proof}

Recall that a Riesz space is said to be {\it $\sigma$-laterally complete} if every disjoint sequence has a supremum. For a set $A$, $\mathbb{R}^A$ is an example of $\sigma$-laterally complete Riesz space.

\begin{proposition}
Let $E$, $F$, $G$ be Riesz spaces. Then we have the following:
\begin{enumerate}
\item If $T\in L_{so}(E,F)$ and $S\in L_{n}(F,G)$, then $ST\in L_{so}(E,G)$.  As a consequence, $L_{so}(E)$ is a left ideal for $L_{n}(E)$. Similarly, $L_{s\sigma o}(E)$ is a left ideal for $L_{c}(E)$.
\item  If $E$ is $\sigma$-Dedekind complete and $\sigma$-laterally complete and  $S\in L_{c}(E,F)$ and $T\in L_{s\sigma o}(F,G)$, then  $TS\in L_{s\sigma o}(E,G)$.  In this case,  $L_{s\sigma o}(E,F)= L_c(E,F)$.
\end{enumerate}
\end{proposition}
\begin{proof}
\begin{enumerate}
\item Let $(x_\alpha)$ be a net in $E$ such that $x_\alpha\xrightarrow{uo}0$. By assumption, $Tx_\alpha\xrightarrow{o}0$. So, $STx_\alpha\xrightarrow{o}0$. Hence,  $ST\in L_{so}(E,G)$.
\item  Let $E$ be a $\sigma$-Dedekind complete and $\sigma$-laterally complete Riesz space. By Theorem 3.9 of \cite{6}, we see that a sequence $(x_n)$ in $E$ is $uo$-null if and only if it is order null. So, if  $(x_n)$ be a sequence  in $E$ such that $x_n\xrightarrow{uo}0$, then $x_n\xrightarrow{o}0$. Thus, $Sx_n\xrightarrow{o}0$ and then $TSx_n\xrightarrow{o}0$. Hence, $TS\in L_{s\sigma o}(E,G)$. Clearly, we have $L_{s\sigma o}(E,F)= L_c(E,F)$. This ends the proof.
\end{enumerate}

\end{proof}

Let $T:E\rightarrow F$ be a positive operator between Riesz spaces. We say that an operator $S:E\rightarrow F$ is dominated by $T$ (or that $T$ dominates $S$ ) whenever $|Sx|\leq T|x|$ holds for each $x\in E$.

\begin{proposition} 
If a positive $so$-continuous operator $T:E\rightarrow F$ dominates $S$, then $S$ is $so$-continuous. 
\end{proposition}
\begin{proof}
Let  $T:E\rightarrow F$ be a positive $so$-continuous operator between Riesz spaces such that $T$ dominates $S:E\rightarrow F$ and let $x_\alpha\xrightarrow{uo}0$ in $E$.  By part (2) of Lemma 1, $|x_\alpha|\xrightarrow{uo}0$. So, by assumption, $T|x_\alpha|\xrightarrow{o}0$  and from the inequality $|Sx|\leq T|x|$ and part (2) of Lemma 1 again,  we have $Sx_\alpha\xrightarrow{o}0$. Hence, $S$ is $so$-continuous.
\end{proof}

For an operator $T:E\rightarrow F$ between two Riesz spaces we shall say that its modulus $|T|$ exists ( or that $T$ possesses a modulus) whenever $|T|:=T\vee (-T)$ exists in the sense that $|T|$ is the supremum of the set $\{-T,T\}$ in $L(E,F)$. If $E$ and $F$ are Riesz spaces with $F$ Dedekind complete, then every order bounded operator $T:E\rightarrow F$ possesses a modulus \cite[Theorem 1.18]{2}. From this discussion it follows that when $E$ and $F$ are Riesz spaces with $F$ Dedekind complete, then each order bounded operator $T:E\rightarrow F$ satisfies $$T^+(x)= \sup\{Ty: 0\leq y\leq x\},\  and$$ $$T^-(x)= \sup\{-Ty: 0\leq y\leq x\} \   $$ for each $x\in E^+$.

\begin{theorem} 
For an order bounded linear functional $f$ on a Riesz space $E$ the following statements are equivalent.
\begin{enumerate}
\item $f$ is $so$-continuous.
\item $f^+$ and $f^-$ are both $so$-continuous.
\item $|f|$ is $so$-continuous.
\end{enumerate}
\end{theorem} 
\begin{proof}
$(1)\Rightarrow (2)$   By Lemma 1, we may assume that $(x_\alpha)\subset E^+$.  Let  $x_\alpha\xrightarrow{uo}0$ and let $(r_\alpha)$ be a net in $\mathbb{R}$ such that $r_\alpha\downarrow 0$.  In view of $f^+x=\sup\{fy :\  0\leq y\leq x\}$,  for each $\alpha$  there exists a net $(y_{\alpha})$ in $E$ with $0\leq y_{\alpha}\leq x_{\alpha}$ such that 
$f^+x_\alpha -r_\alpha \leq fy_\alpha$. So, $f^+x_\alpha \leq fy_\alpha+r_\alpha$. Since $x_\alpha\xrightarrow{uo}0$, we have $y_{\alpha}\xrightarrow{uo}0$. Thus, by assumption,  $fy_{\alpha}\xrightarrow{o}0$. It follows from $f^+x_\alpha \leq (fy_\alpha+r_\alpha)\xrightarrow{o}0$ that  $f^+x_\alpha\xrightarrow{o}0$.  Hence, $f^+$ is $so$-continuous. Now, as $f^-=(-f)^+$, we conclude that $f^-$ is also $so$-continuous.\\
$(2)\Rightarrow (3)$ Follows from the identity $|f|=f^++f^-$. \\
$(3)\Rightarrow (1)$ Follows immediately from Proposition 3 by observing  that $|f|$ dominates $f$. 
\end{proof}

\begin{remark}
One can easily formulate by himself the analogue of Theorem 1 for $s\sigma o$-continuous operators.
\end{remark}

Recall that a subset $A$ of a Riesz space is said to be order closed whenever $(x_\alpha)\subset A$ and $x_\alpha \xrightarrow{o}x$ imply $x\in A$. An order closed ideal is referred to as a band. Thus, an ideal $A$ is a band if and only if  $(x_\alpha)\subset A$ and $0\leq x_\alpha\uparrow x$ imply $x\in A$.  In the next theorem we show that $L_{so}(E,\mathbb{R})$ and $L_{s\sigma o}(E,\mathbb{R})$ are both bands of $E^\sim$. The details follow.

\begin{theorem}
If $E$ is a Riesz space, then  $L_{so}(E,\mathbb{R})$ and $L_{s\sigma o}(E,\mathbb{R})$ are both bands of $E^\sim$.
\end{theorem}
\begin{proof}
We only show that $L_{so}(E,\mathbb{R})$ is a band of $E^\sim$. That $L_{s\sigma o}(E,\mathbb{R})$ is a band can be proven in a similar manner. Note first that if $|g|\leq |f|$ holds in $E^\sim$ with $f\in L_{so}(E,\mathbb{R})$, then from Theorem 1  it follows that $g\in L_{so}(E,\mathbb{R})$. That is $L_{so}(E,\mathbb{R})$ is an ideal of $E^\sim$. To see that the ideal $L_{so}(E,\mathbb{R})$ is a band, let $0\leq f_\lambda\uparrow f$ in $E^\sim$ with $(f_\lambda)\subset L_{so}(E,\mathbb{R})$, and let $0\leq x_\alpha\xrightarrow{uo}0$ in $E$. Then for each fixed $\lambda$  we have $$0\leq f(x_\alpha)=((f-f_\lambda)(x_\alpha) +f_\lambda(x_\alpha))\xrightarrow{o}0.$$ So, $f(x_\alpha)\xrightarrow{o}0$.  Thus, $f\in  L_{so}(E,\mathbb{R})$, and the proof is finished. 
\end{proof}

\begin{problem}
Can we find a $so$-continuous operator $T:E\rightarrow F$ between  Riesz spaces whose modulus is not $so$-continuous?
\end{problem}


\vskip 0,65 true cm

\end{document}